\documentclass[11pt, psamsfonts]{amsart}

\usepackage[a4paper,margin=4cm]{geometry}
\usepackage{amsfonts,amssymb,amscd}
\usepackage{graphicx,mathrsfs} 
\usepackage{subfigure}
\newtheorem{theorem}{Theorem}[section]

\newtheorem{lemma}[theorem]{Lemma}
\newtheorem{corollary}[theorem]{Corollary}

\newtheorem{definition}[theorem]{Definition}

\theoremstyle{remark}
\newtheorem{conjecture}[theorem]{\bf Conjecture}
\newtheorem{remark}[theorem]{\bf Remark}
\newtheorem{example}[theorem]{\bf Example}

\renewcommand{\leq}{\leqslant}
\renewcommand{\geq}{\geqslant}

\newcommand{\ptl}{\partial}

\newcommand{\rr}{\mathbb{R}}

\newcommand{\nn}{\mathbb{N}}

\DeclareMathOperator{\inte}{int}


\numberwithin{equation}{section}

\begin{document}

\title[Minimizing subdivisions for the maximum relative diameter]
{Subdivisions of rotationally symmetric planar convex bodies minimizing
the maximum relative diameter}

\author[A. Ca\~nete]{Antonio Ca\~nete}
\address{Departamento de Matem\'atica Aplicada I \\ Universidad de Sevilla}
\email{antonioc@us.es}

\author[U. Schnell]{Uwe Schnell}
\address{Faculty of Mathemics and Sciences \\ University of Applied Sciences Zittau/G{\"or}litz}
\email{uschnell@hszg.de}

\author[S. Segura Gomis]{Salvador Segura Gomis}
\address{Departamento de An\'alisis Matem\'atico \\ Universidad de Alicante}
\email{salvador.segura@ua.es}

\date{\today}

\begin{abstract}
In this work we study subdivisions of $k$-rotationally symmetric planar convex bodies
that minimize the maximum relative diameter functional. For some particular subdivisions called $k$-partitions,
consisting of $k$ curves meeting in an interior vertex, we prove that the so-called \emph{standard $k$-partition}
(given by $k$ equiangular inradius segments) is minimizing for any $k\in\nn$, $k\geq 3$. For general subdivisions,
we show that the previous result only holds for $k\leq 6$. We also study the optimal set for this problem,
obtaining that for each $k\in\nn$, $k\geq 3$, it consists of the intersection of the unit circle
with the corresponding regular $k$-gon of certain area.
Finally, we also discuss the problem for planar convex sets and large values of $k$,
and conjecture the optimal $k$-subdivision in this case.
\end{abstract}

\maketitle

\section{Introduction}

The study of the classical geometric magnitudes associated to a planar compact convex set
(perimeter, area, inradius, circumradius, diameter and minimal width) is one of the main points
of interest for Convex Geometry, and these magnitudes have been considered since ancient times.
They appear in the origin of this mathematical field,
and in fact, there is a large variety of related problems,
providing a more complete understanding of these geometric measures.

One of the most common problems involving these magnitudes
consists of finding precise relations between some of them (and not only in the convex framework).
For instance, the well-known isoperimetric inequality relates the perimeter and the area functionals,
stating that, among all planar sets with fixed enclosed area,
the circle is the one with the least possible perimeter \cite{S,O}.
In the same direction, the isodiametric inequality
asserts that the circle is also the planar compact convex set of prescribed area
with the minimum possible diameter \cite{bieberbach}.
Another example is given by the equilateral triangle,
which is the planar compact convex set with minimum enclosed area for fixed minimal width \cite{pa}.
We point out that these relations are usually established by means of appropriate general inequalities,
with the characterization of the equality case providing then the optimal set for the corresponding functionals
(that is, the set attaining the minimum possible value for the functional).
A nice detailed description of the relations between two (and three) of these classical magnitudes can be found in~\cite{sa}.

Apart from the above approach, these geometric functionals can be studied in a large range of different related problems.
One of these problems, which has been considered in several works during the last years, is the so-called \emph{fencing problem}, or more generally, the \emph{partitioning problem}. For these problems, the main interest is determining the division of the given set into $k$ connected subsets (of equal or unequal areas) which minimizes (or maximizes) a particular geometric magnitude \cite{cfg}. The relative isoperimetric problem is one of these questions, where the goal is minimizing the perimeter of the division curves for fixed enclosed areas. For this problem, the general geometric properties satisfied by the minimizing divisions can be deduced by using a variational approach, and particular results characterizing the solutions have been obtained for some regular polygons and for the circle, see \cite{tomonaga, bleicher, cr, cox}.

Regarding the diameter functional, the partitioning problem can be studied by means of the \emph{maximum relative diameter} functional.
We recall that for a planar compact convex set $C$ and a decomposition of $C$ into $k$ connected subsets $C_1,\ldots,C_k$,
the maximum relative diameter associated to the decomposition
is defined as the maximum of the diameters of the subsets $C_i$, for $i=1,\dots,k$
(this means that it measures the largest distance in the subsets provided by the decomposition).
In this setting, the partitioning problem seeks the decomposition of $C$
into $k$ subsets with the least possible value for the maximum relative diameter.
This problem has been recently studied for $k=2$ and $k=3$
under some additional symmetry hypotheses \cite{mps,trisecciones},
and the results therein constitute the main motivation for our work.

In \cite{mps}, the previous problem is treated for \emph{centrally symmetric} planar compact convex sets,
and for decompositions into two subsets of equal areas,
proving that the minimizing division for the maximum relative diameter
is always given by a line segment passing through the center of symmetry of the set~\cite[Prop.~4]{mps}.
We point out that a more precise characterization of such a minimizing segment is not provided.
Furthermore, the optimal set for this problem (that is, the centrally symmetric planar compact convex set attaining
the minimum possible value for the maximum relative diameter) is also determined up to dilations,
consisting of a certain intersection of two unit circles and a strip bounded by two parallel lines, see \cite[~Th.~5]{mps}.

In \cite{trisecciones}, the problem is considered for \emph{$3$-rotationally symmetric}
planar compact convex sets, and for divisions into three subsets of equal areas.
The main result establishes that the so-called \emph{standard trisection},
determined by three equiangular inradius segments,  minimizes the maximum relative diameter \cite[Th.~3.5 and Prop.~5.1]{trisecciones} for any set of that class.
In addition, the optimal set for this problem is also characterized, up to dilations,
as the intersection of the unit circle with a certain equilateral triangle \cite[Th.~4.7]{trisecciones}.
We remark that, although the results in \cite{trisecciones}
are stated for divisions into three subsets of \emph{equal areas},
all of them also hold in the case of unequal areas.

This paper is inspired precisely in these last two references
involving the maximum relative diameter~\cite{mps, trisecciones}.
As a natural continuation, we shall focus on the class $\mathcal{C}_k$ of
$k$-rotationally symmetric planar convex bodies (that is, compact convex sets) for each $k\in\nn$, $k\geq 3$.
And for each set in $\mathcal{C}_k$, we shall investigate the divisions into $k$ connected subsets
minimizing the maximum relative diameter functional.
In this setting we shall consider two different types of divisions,
namely \emph{$k$-subdivisions} and \emph{$k$-partitions}.
For a given set $C\in\mathcal{C}_k$,
a $k$-partition of $C$ will be a decomposition of $C$ into $k$ connected subsets
by $k$ curves starting at different points of $\ptl C$,
and with all of them meeting at an interior point in $C$.
On the other hand, a $k$-subdivision of $C$ will be a general decomposition of $C$
into $k$ connected subsets, with no additional restrictions.

The goal of this paper is studying the minimizing $k$-partitions and minimizing $k$-subdivisions for the maximum relative diameter, for any $k$-rotationally symmetric planar convex body.
In fact, we shall see that the so-called \emph{standard $k$-partition}
(consisting of $k$ inradius segments symmetrically placed) is a minimizing $k$-partition for the maximum relative diameter,
for any $k\in\nn$, $k\geq 3$ (Theorem~\ref{th:minimizing}),
being also minimizing among $k$-subdivisions only when $k\leq 6$
(Theorem~\ref{th:minimizing-subdivision}).
For $k>7$ we shall provide some examples showing that
the standard $k$-partition is not a minimizing $k$-subdivision,
which suggests that, in those cases, a complete characterization is a difficult question.
In addition, we point out that the uniqueness of the minimizing $k$-partition
does not hold for this problem, as explained in Subsection~\ref{subsec:uniqueness},
since proper slight deformations of a minimizing $k$-partition are also minimizing.

Another interesting result in this paper is the characterization of the \emph{optimal bodies} for this problem.
For each $k\in\nn$, $k\geq3$, we are able to determine the $k$-rotationally symmetric planar convex body
attaining the minimum possible value for the maximum relative diameter, when considering $k$-partitions.
Taking into account the previous Theorem~\ref{th:minimizing}, for each set of this class,
that minimum value will be provided by the corresponding standard $k$-partition.
In Section~\ref{se:optimal} we give the description of these optimal bodies for each $k\geq 3$,
which, in general, consist of the intersections of the unit circle with a certain regular $k$-gon, up to dilations.

We remark that the case $k=2$ (treating the decompositions into two subsets
for centrally symmetric planar convex bodies and previously studied in \cite{mps})
presents some differences with respect to the other situations,
as explained in Remark~\ref{re:2}, and cannot be treated as for $k\geq 3$.
The main reason for this particular behaviour lies in the following technical fact:
the associated standard $2$-partition of a set of this class, which is provided
by two inradius segments placed symmetrically, is just a line segment
passing through the center of symmetry of the set,
and so the computation of the maximum relative diameter associated to this standard $2$-partition
cannot be done by using Lemma~\ref{le:dM},
which is fundamental when $k\geq3$.

We have organized this paper as follows. Section~\ref{sec:preliminaries}
contains the precise definitions and the setting of the problem,
and in Section~\ref{sec:standard} we describe the standard $k$-partition
for any $k$-rotationally symmetric planar convex body.
This $k$-partition is constructed by applying successively the existing rotational symmetry
to an inradius segment of the set.
We remark that Lemma~\ref{le:dM} establishes how to compute the maximum relative diameter for the standard $k$-partition.

In Section~\ref{sec:main} we obtain the main results of this paper. Due to Lemmata~\ref{le:dM},~\ref{le:cota1} and~\ref{le:cota2}, we prove our Theorem~\ref{th:minimizing}:
\begin{quotation}
\emph{For any $k$-rotationally symmetric planar convex body,
the associated standard $k$-partition is a minimizing $k$-partition for the maximum relative diameter.}
\end{quotation}
Furthermore, regarding general $k$-subdivisions, we obtain Theorem~\ref{th:minimizing-subdivision}:
\begin{quotation}
\emph{For any $k$-rotationally symmetric planar convex body, the associated standard $k$-partition is a minimizing $k$-subdivision for the maximum relative diameter, when $k\leq6$.}
\end{quotation}
\begin{figure}[h]
  \includegraphics[width=0.6\textwidth]{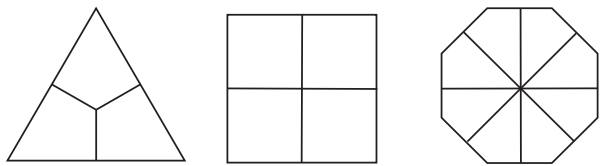}\\
  \caption{The standard $3$-partition for the equilateral triangle, the standard $4$-partition for the square,
  and the standard $8$-partition for the regular octagon}
\end{figure}

When $k\geq7$, Examples~\ref{ex:7} and \ref{ex:8} show that the standard $k$-partition
is not a minimizing $k$-subdivision in general, suggesting that the characterization of such minimizing $k$-subdivisions
is a difficult task which will depend on the particular considered set.
We also point out that, for $k=3$, the minimizing $3$-partition and $3$-subdivision in this setting
have been previously obtained, see \cite[Th.~3.5 and Prop.~5.1]{trisecciones}.

Moreover, in Subsection~\ref{subsec:uniqueness} we discuss the uniqueness of the solutions for this problem.
We will show with some examples that slight modifications of a minimizing $k$-partition in a proper way
(preserving the value of the maximum relative diameter) will produce minimizing $k$-partitions as well,
and so we do not have uniqueness in this setting (as in most of the problems involving the diameter functional).
This subsection concludes by posing some necessary conditions for a given $k$-partition to be minimizing,
derived from Lemma~\ref{le:dM} and Theorem~\ref{th:minimizing}.

In Section~\ref{se:optimal} we study the \emph{optimal bodies} for this problem.
More precisely, for each $k\in\nn$, $k\geq3$, we characterize
the $k$-rotationally symmetric planar convex body for fixed area
attaining the lowest possible value for the maximum relative diameter when considering $k$-partitions.
Notice that, in view of Theorem~\ref{th:minimizing}, we have to focus on the standard $k$-partitions,
and as a consequence of Lemma~\ref{le:dM} we are able to prove our Theorem~\ref{th:optimal}:
\begin{quotation}
\emph{For any $k\in\nn$, $k\geq 3$, the minimum value for the maximum relative diameter for $k$-partitions
in the class of $k$-rotationally symmetric planar convex bodies is uniquely attained
by the standard $k$-partition associated to the intersection of the unit circle and the regular $k$-gon
with inradius equal to $1/(2\sin(\pi/k))$, up to dilations.}
\end{quotation}
The previous Theorem~\ref{th:optimal} characterizes the optimal body,
for each $k\in\nn$, $k\geq3$, as the intersection of the unit circle with a particular regular $k$-gon.
In Subsection~\ref{re:description} we analyze these intersections in each case, obtaining the precise optimal sets:
for $k=3$ and $k=5$, the optimal bodies are respectively an equilateral triangle and a regular pentagon
\emph{with rounded vertices}, for $k=4$ we have a square, and for $k\geq6$ the best set in this setting
is the unit circle. We remark that the optimal body for $k=3$ was already characterized in~\cite[Th.~4.7]{trisecciones}
by means of a constructive procedure.

\begin{figure}[h]
  \includegraphics[width=0.85\textwidth]{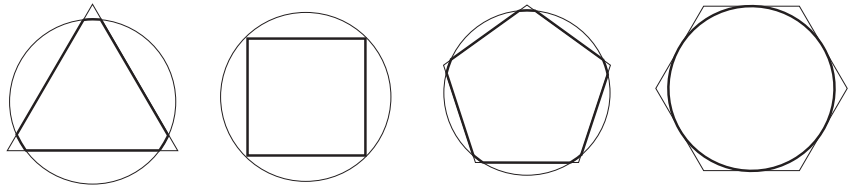}\\
  \caption{Optimal sets for $k=3$, $k=4$, $k=5$, and $k\geq6$}
\end{figure}

Finally, in Section~\ref{sec:large} we complete our study by considering large values of $k\in\nn$,
discussing our problem for $k$-subdivisions (focusing on planar convex bodies).
Recall that Theorem~\ref{th:minimizing} establishes that the standard $k$-partitions
are minimizing $k$-partitions for any $k\in\nn$, $k\geq 3$, and minimizing $k$-subdivisions only when $k\leq6$,
for any $k$-rotationally symmetric planar convex body.
In Corollary~\ref{co:cota} we obtain a lower bound for the maximum relative diameter functional
when considering $k$-subdivisions of a planar convex body,
involving a negligible addend when $k$ is large enough.
However, we think that such lower bound can be improved,
and some reasonings in this section lead us to Conjecture~\ref{conj:infinito}:
for large values of $k\in\nn$, a minimizing $k$-subdivision will be given
by a decomposition induced by a planar tiling of regular hexagons.

\section{Preliminaries}
\label{sec:preliminaries}

In this paper we will consider, for any given $k\in\nn$, $k\geq 3$, the class $\mathcal{C}_k$
of $k$-rotationally symmetric planar convex bodies (assuming therefore the compactness of the sets).
We recall that, in this setting, each set $C$ in the class $\mathcal{C}_k$ has a center of symmetry $p$,
and then, the $k$-rotationally symmetry means that $C$ is invariant
under the rotation centered at $p$ with angle $\varphi_k:=2\pi/k$.
Some examples of sets belonging to $\mathcal{C}_k$ are the regular $k$-gon and the Reuleaux $k$-gon.

\begin{figure}[h]
  \includegraphics[width=0.8\textwidth]{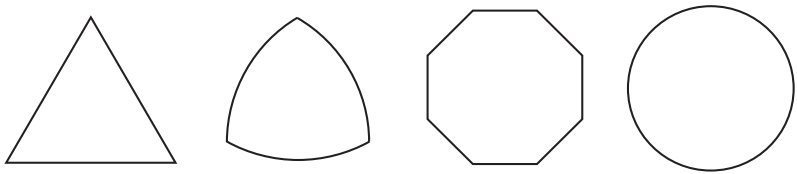}\\
  \caption{Some examples of $k$-rotationally symmetric planar convex bodies:
  the equilateral and Reuleaux triangles are $3$-rotationally symmetric,
  the regular octogon is $8$-rotationally symmetric (and also $4$-rotationally symmetric),
  and the circle is $k$-rotationally symmetric for any $k\in\nn$. }
\end{figure}

Given a planar compact set $C$, we shall denote by $r(C)$ the \emph{inradius} of $C$
(that is, the radius of the largest ball contained in $C$),
and by $R(C)$ the \emph{circumradius} of $C$ (radius of the smallest ball containing $C$).
If no confusion may arise, we shall simply denote them by $r$ and $R$, respectively.
For $k$-rotationally symmetric planar convex bodies, we have the following Lemma,
which provides a metric characterization for the inradius and the circumradius.

\begin{lemma}
\label{le:rR}
Let $C\in\mathcal{C}_k$, and let $p$ be its center of symmetry.
Then, the inradius $r$ coincides with the minimal Euclidean distance between $p$ and a point in $\ptl C$,
and the circumradius $R$ coincides with the maximal Euclidean distance between $p$ and a point in $\ptl C$.
\end{lemma}

\begin{proof}
It is known that the set $I$ of centers of all largest balls contained in $C$
is a convex set with empty interior \cite{eugenia}.
As $C$ is $k$-rotationally symmetric, this implies that $I=\{p\}$
(note that for any $c\in I$, $c\neq p$, the images of $c$ by the existing rotations would
belong to $I$, and so $p$ would be an interior point of $I$).
Then there is a unique largest ball contained in $C$, which is centered at $p$.
This immediately yields the statement for the inradius.

On the other hand, as the smallest ball containing $C$ is unique, its center has to coincide with $p$
(otherwise, an analogous rotation argument will contradict that uniqueness). This implies trivially
the statement for the circumradius.

\end{proof}

\begin{figure}[h]
    \includegraphics[width=0.28\textwidth]{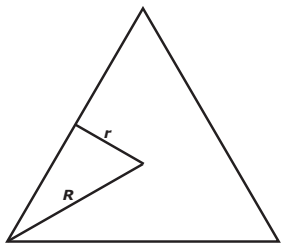}\\
  \caption{The inradius $r$ and the circumradius $R$ of an equilateral triangle.}
\end{figure}

We now define the notions of $k$-subdivision and $k$-partition of a planar compact set,
which will be considered along this work.
Notice that these definitions can be generally done 
without assuming additional hypotheses regarding convexity or symmetry.

\begin{definition}
Let $C$ be a planar compact set. A $k$-subdivision $S$ of $C$ is a decomposition of $C$ into $k$ connected subsets $\{C_1,\ldots,C_k\}$, not necessarily of equal areas, satisfying:
\begin{itemize}
\item[i)] $\displaystyle{C=\bigcup_{i=1}^k C_i}$,
\item[ii)] $\inte(C_i)\cap \inte(C_j)=\emptyset$, for $i,j\in\{1,\ldots,k\}$, $i\neq j$.
\end{itemize}

In addition, if a $k$-subdivision $S$ is given by $k$ curves meeting at a point $c\in\inte(C)$,
and the other endpoints of the curves are different points in $\ptl C$,
then $S$ will be called a $k$-partition of $C$.
In this case, $c$ is called the common point of the $k$-partition and the endpoints of the curves meeting $\ptl C$
will be called the endpoints of the $k$-partition.
\end{definition}

We stress that $k$-subdivisions are general decompositions of the original set $C$ into $k$ connected subsets,
and that $k$-partitions are $k$-subdivisions with a particular topological structure: they consist of $k$ curves
converging to an interior point, each of them meeting the boundary of $C$ at a different point.

\begin{figure}[h]
  \includegraphics[width=0.66\textwidth]{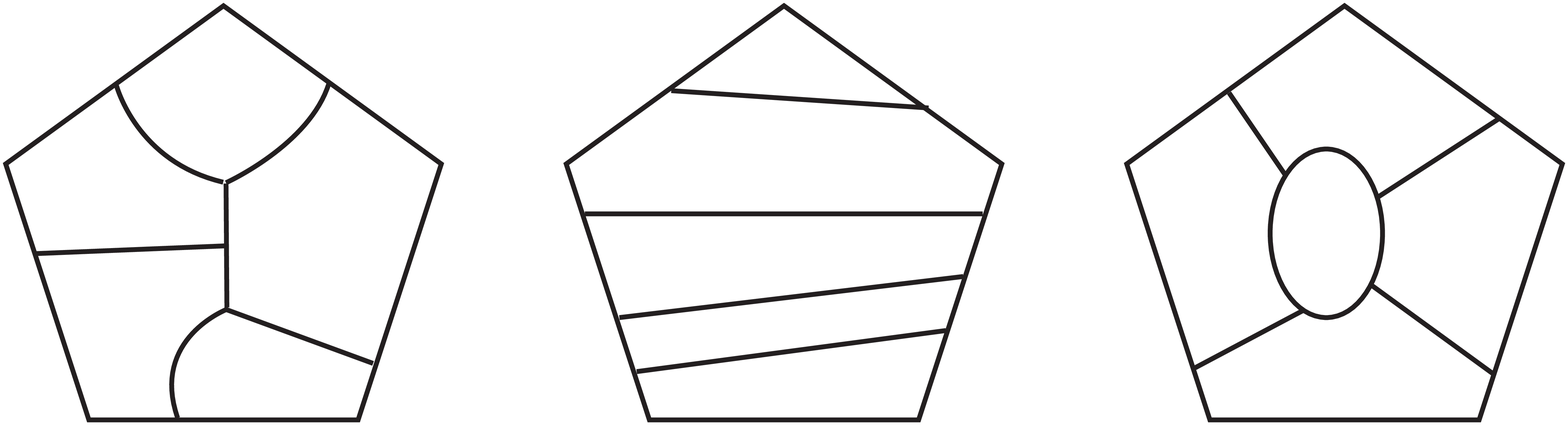}\\
  \caption{Three different $5$-subdivisions for the regular pentagon.}
\end{figure}

\begin{figure}[h]
  \includegraphics[width=0.66\textwidth]{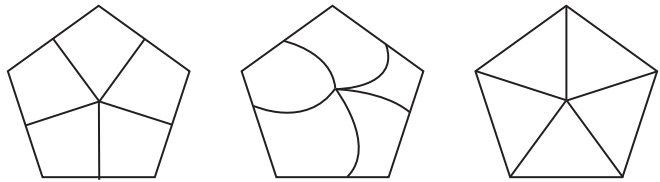}\\
  \caption{Three different $5$-partitions for the regular pentagon.}
\end{figure}

\begin{remark}
In this work, we are focusing on the class $\mathcal{C}_k$ of $k$-rotationally symmetric planar convex bodies
since the lack of the rotational symmetry prevents a general study of this question.
Moreover, for any set $C\in\mathcal{C}_k$, its center of symmetry is a remarkable point in the interior of $C$,
and so it seems natural to consider $k$-partitions, which are precisely
the $k$-subdivisions determined by interior points of $C$.
\end{remark}

We now recall the definition of the maximum relative diameter for a $k$-subdivision of a planar compact set.

\begin{definition}
Let $C$ be a planar compact set, and let $S$ be a $k$-subdivision of $C$ into subsets $\{C_1,\ldots,C_k\}$.
We define the \emph{maximum relative diameter} associated to $S$ as
\[
d_M(S,C)=\max\{D(C_i):i=1,\ldots,k\},
\]
where $D(C_i)=\max\{d(x,y):x,y\in C_i\}$ denotes the Euclidean diameter of $C_i$.
\end{definition}

\begin{remark}
Along this work, we shall usually write $d_M(S)$ instead of $d_M(S,C)$, if no confusion may arise.
\end{remark}

\begin{remark}
For a given planar compact set $C$ and a $k$-subdivision $S$ of $C$,
the existence of $d_M(S,C)$ is assured due to the continuity of the Euclidean distance $d:\rr^2\to\rr$,
and the compactness of $C$.
\end{remark}

\begin{remark}
By a standard 
argument it is clear that $d_M(S,C)$ is attained
by points lying in the boundary of one of the subsets $C_i$ determined by
the $k$-subdivision $S$.
Moreover, it is well-known that the diameter of a convex polygon is given by the distance between two of its vertices.
\end{remark}

In the forthcoming sections, we shall investigate the $k$-partitions and the $k$-subdivisions
of a given $k$-rotationally symmetric planar convex body which minimize the maximum relative diameter functional.
In other words, for a given set $C\in\mathcal{C}$, we shall look for
\[
\min\{d_M(P,C): P\text{ is a $k$-partition of $C$\,}\},
\]
and
\[
\min\{d_M(S,C): S\text{ is a $k$-subdivision of $C$\,}\}.
\]

The $k$-partitions and $k$-subdivisions attaining those minimum values will be called \emph{minimizing}.
As usual in this kind of optimization problems involving the diameter functional,
we point out that the uniqueness of the minimizing $k$-subdivision is not expected,
as we shall see in Subsection~\ref{subsec:uniqueness}.

\section{Standard $k$-partition}
\label{sec:standard}

In this section we shall describe a particular $k$-partition for the sets of our class $\mathcal{C}_k$,
called the \emph{standard $k$-partition}, by means of a simple geometrical intrinsic construction.
This $k$-partitions will play an important role along this work.

For $C\in\mathcal{C}_k$, let $x_1\in\ptl C$ be a point in $\ptl C$
with minimal distance to the center of symmetry $p$.
In view of Lemma~\ref{le:rR}, this implies that $d(p,x_1)$ is the inradius of $C$,
and so the segment $\overline{p\,x_1}$ will be referred to as \emph{an inradius segment} of $C$
(notice that for regular polygons, such a segment is called apothem).
By applying $k-1$ times the rotation centered at $p$ with angle $\varphi_k=2\pi/k$ to the point $x_1$,
we shall finally obtain a set of points $\{x_1,\ldots, x_k\}\subset\ptl C$, all of them minimizing the distance to $p$.
If we join these points to $p$, we shall get $k$ inradius segments dividing $C$ into $k$ connected congruent subsets.
This $k$-partition will be called \emph{standard k-partition} of $C$, and will be denoted by $P_k(C)$, or simply $P_k$.

\begin{figure}[h]
  \includegraphics[width=0.7\textwidth]{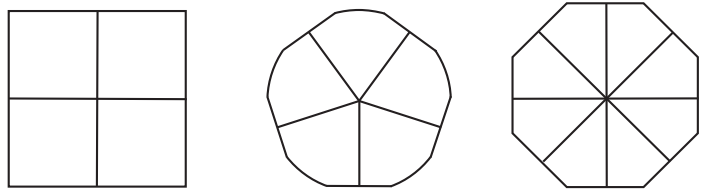}\\
  \caption{Standard $k$-partitions for different $k$-rotationally symmetric planar bodies}
\end{figure}

\begin{remark}
In general, the previous construction of the standard $k$-partition
associated to a $k$-rotationally symmetric planar convex body $C$ is not uniquely determined.
In fact, if there are more than $k$ points in $\ptl C$ attaining the minimal distance
to the center of symmetry of $C$, then the associated standard $k$-partition is not unique
(for instance, this happens for the standard $3$-partition of the regular hexagon,
or for the standard $k$-partition of the circle for any $k\geq 2$).
In any case, all the standard $k$-partitions associated to $C$ will be \emph{$k$-rotationally congruent} by construction.
\end{remark}

For the standard $k$-partition of a given set $C$ in $\mathcal{C}_k$,
the associated maximum relative diameter can be easily computed in terms
of the inradius and the circumradius of $C$, as proved in the following Lemma~\ref{le:dM}.

\begin{lemma}
\label{le:dM}
Let $C$ be a $k$-rotationally symmetric planar convex body, with inradius $r$ and circumradius $R$.
Let $P_k$ be the standard $k$-partition associated to $C$. Then,
\[
d_M(P_k,C)=\max\{R,2\,r\sin(\varphi_k/2)\}.
\]
\end{lemma}

\begin{remark}
Recall that $\varphi_k/2=\pi/k$.
\end{remark}

\begin{proof}
We can assume that $d_M(P_k,C)=D(C_1)$, where $C_1$ is one of the congruent subsets determined by $P_k$.
Call $x_1,x_2\in\ptl C_1$ the endpoints of the two inradius segments bounding $C_1$.
Any pair of points realizing $D(C_1)$ cannot be contained in the relative interior of an inradius segment
(in that case, the diameter could be clearly increased by considering some proper point along the segment).
As the maximal distance between two points on $\ptl C\cap\ptl C_1$ is $d(x_1,x_2)=2\,r\sin(\pi/k)$,
and the maximal distance between $p$ and $\ptl C$ equals $R$, the statement follows.
\end{proof}

The following examples 
show that both possibilities from Lemma~\ref{le:dM} may occur.

\begin{example} (Circle)
\label{ex:dMcircle}
Any circle is a $k$-rotationally symmetric planar convex body for any $k\in\nn$, $k\geq 2$,
whose inradius $r$ coincides with its circumradius $R$. Consequently, in view of Lemma~\ref{le:dM},
it is immediate checking that $d_M(P_k)=2\,r\sin(\varphi_k/2)$ for $k\in\{3,4,5,6\}$,
and $d_M(P_k)=R$ for $k\geq 6$.
\end{example}

\begin{example}(Regular polygons)
\label{ex:rR}
For a regular $k$-gon $E_k$, Lemma~\ref{le:dM} yields that $d_M(P_k,E_k)=R(E_k)$,
for $k\in\nn$, $k\geq 3$, due to the well-known relation $r(E_k)=R(E_k)\cos(\varphi_k/2)$.
\end{example}

\begin{remark}
\label{re:2}
Lemma~\ref{le:dM} does not hold for $k=2$ (it can be easily checked by considering the square or the circle).
The reason of this fact is that the standard $2$-partition $P_2$
will divide any centrally symmetric planar convex body $C$ into two subsets
whose boundaries do not contain the center of symmetry $p$ as a \emph{vertex}.
Notice that $p$ will belong to the line segment $\overline{x_1x_2}$ (being $x_1,x_2$ the endpoints of $P_2$),
and so $p$ will not be relevant for computing $d_M(P_2,C)$. In fact, it can be checked that
\[
d_M(P_2,C)=\max\{d(x_1,x): x\in\ptl C\}.
\]
The main difficulty in this case is that the quantity $d(x_1,x)$, when $x\in\ptl C$,
cannot be estimated in general,
which prevents using the arguments from this paper for $k=2$.
We recall that this case was studied in \cite{mps},
where it is proved that the minimizing $2$-partition for $d_M$
is given by a line segment passing through the center of symmetry of the set,
without a more precise description \cite[Prop.~4]{mps}.
\end{remark}

\section{Main results}
\label{sec:main}

In this section, we aim to compare the maximum relative diameters
associated to the standard $k$-partition $P_k$ and an arbitrary $k$-partition $P$
of a given $k$-rotationally symmetric planar convex body $C$, when $k\geq 3$.
By means of Lemmata~\ref{le:cota1} and~\ref{le:cota2},
we will conclude that $d_M(P_k,C)\leq d_M(P,C)$,
and so the standard $k$-partition always provides a minimizing $k$-partition
for the maximum relative diameter functional.
Furthermore, if we consider general $k$-subdivisions,
we will see that the same previous result holds only when $k\leq6$ (Theorem~\ref{th:minimizing-subdivision}),
showing also some counterexamples for $k\geq7$.
Finally, in Subsection~\ref{subsec:uniqueness} we will see that
the uniqueness of the minimizing $k$-partition cannot be expected for this problem.

\begin{lemma}
\label{le:cota1}
Let $C$ be a $k$-rotationally symmetric planar convex body, with circumradius $R$, for $k\geq 3$.
Let $P$ be a $k$-partition of $C$. Then $d_M(P)\geq R$.
\end{lemma}

\begin{proof}
Let $c$ be the common point of $P$, and $C_1,\ldots,C_k$ the corresponding subsets of $C$ given by $P$.
Note that $c\in\bigcap_{i=1}^k C_i$. Call $p$ the center of symmetry of $C$,
and let $y_1,\ldots,y_k$ be points in $\ptl C$ such that $d(p,y_i)=R$
(we can assume they are placed $k$-symmetrically along $\ptl C$).
Then, it is not difficult to check that $p=\sum_{i=1}^k y_i/k$
is the unique minimum for the functional $S:C\to\rr$ given by
$$\displaystyle{S(x)=\sum_{i=1}^k d(x,y_i)^2},$$
and so
$$\displaystyle{S(c)=\sum_{i=1}^k} d(c,y_i)^2\geq S(p)=k R^2.$$
This necessarily implies that some $y_j$ satisfies $d(c,y_j)\geq R$,
yielding $d_M(P)\geq d(c,y_j)\geq R$, as desired.
\end{proof}

The above Lemma~\ref{le:cota1} also holds for $k$-subdivisions, when $k\leq 6$.
Notice that the proof is different from the one of Lemma~\ref{le:cota1},
since $k$-subdivisions do not have a common interior point as $k$-partitions.

\begin{lemma}
\label{le:cota1-subd}
Let $C$ be a $k$-rotationally symmetric planar convex body,
with circumradius $R$, for $3\leq k\leq 6$. Let $S$ be a $k$-subdivision of $C$.
Then $d_M(S)\geq R$.
\end{lemma}

\begin{proof}
Let $y_1,\ldots,y_k$ be points in $\ptl C$ such that $d(p,y_i)=R$.
We can assume they are placed $k$-symmetrically along $\ptl C$,
and so they will determine a regular $k$-gon $E_k$ with edge length equal to
$2R\sin(\pi/k)\geq 2R\sin(\pi/6)=R$, since $k\leq 6$.
Then we have that $p,y_1,\ldots,y_k$ are $k+1$ points
with mutual distances greater than or equal to $R$.
Since at least two of these points will be contained in a same subset
given by the subdivision $S$, it follows that $d_M(S)\geq R$.
\end{proof}

\begin{remark}
\label{re:7}
The previous Lemma~\ref{le:cota1-subd} does not hold when $k>6$.
For instance, for $k=7$, consider a regular heptagon $H$ with circumradius $R=1$ and center of symmetry $p$,
and the $7$-subdivision $S$ shown in Figure~\ref{7}
(which is a modification of the $7$-partition of $H$ whose endpoints
are the vertices of $H$ with $p$ as common point).
This $7$-subdivision is given by eleven line segments meeting in threes at five inner points
(placed at the same distance from $p$), with congruent subsets $H_2$, $H_7$,
congruent subsets $H_3$, $H_4$, $H_5$, $H_6$, and $p\in H_1$.
In this case, we have checked that $d_M(S)=D(H_2)=d(a,b)=0.9892<R=1$,
and so Lemma~\ref{le:cota1-subd} is not satisfied.
\end{remark}

\begin{figure}[h]
   \includegraphics[width=0.88\textwidth]{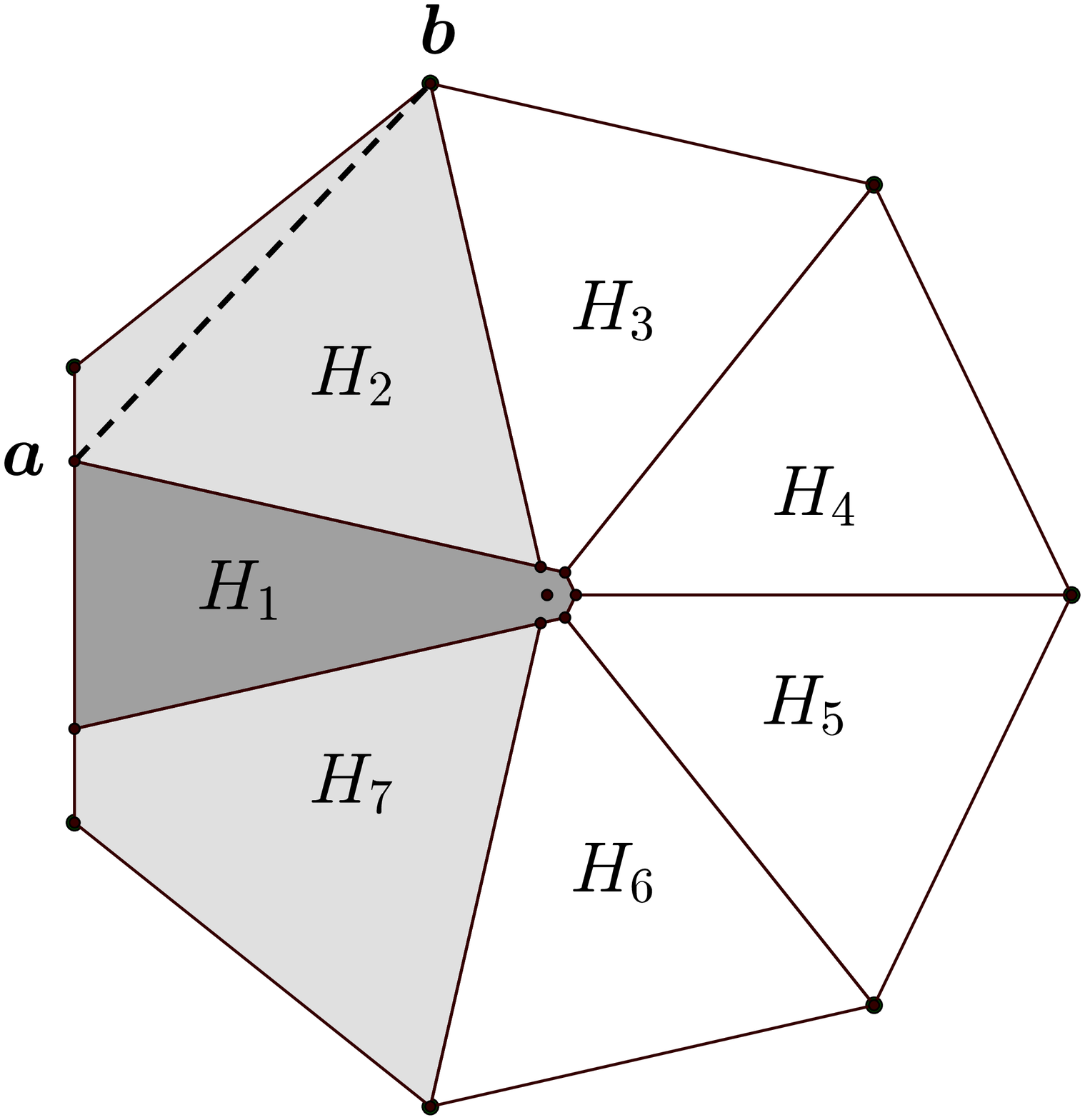}\\
  \caption{A $7$-subdivision of the regular heptagon with 
  maximum relative diameter striclty less than the circumradius}
  \label{7}
\end{figure}

The following Lemma~\ref{le:cota2} is true not only for $k$-partitions, but also for general $k$-subdivisions.

\begin{lemma}
\label{le:cota2}
Let $C$ be a $k$-rotationally symmetric planar convex body, with inradius $r$, for $k\geq3$ .
Let $S$ be a $k$-subdivision of $C$. Then $d_M(S)\geq 2\,r\sin(\pi/k)$.
\end{lemma}

\begin{proof}
Let $x_1,\ldots,x_k$ be points in $\ptl C$ with $d(p,x_i)=r$.
We can assume that they are placed $k$-symmetrically,
and so they determine a regular $k$-gon with edge length $2\,r\sin(\pi/k)$.
Call $C_1,\ldots,C_k$ the subsets of $C$ given by the $k$-subdivision $S$.

If two (consecutive) points $x_\alpha, x_\beta$ belong to a same subset $C_i$,
then $d_M(S)\geq D(C_i)\geq d(x_\alpha,x_\beta)=2\,r\sin(\pi/k)$, and the result follows.
So we will assume that each point $x_i$ only belongs to $C_i$, for $i=1,\ldots,k$.

Let $p$ be the center of symmetry of $C$, and let $K$ be the circle with center $p$ and radius $r$.
Clearly $x_1,\ldots,x_k\in\ptl K$.
Since any pair of consecutive points $x_{i-1}$, $x_{i}$
are contained in different subsets $C_{i-1}$, $C_{i}$, respectively,
there must be a point $q_i\in\ptl K$ between them, with $q_i\in C_{i-1}\cap C_{i}$.
Let $\alpha_i$ denote the angle at $p$ determined by the segments $\overline{p\,q_i}$ and $\overline{p\,q_{i+1}}$.
We have that $\alpha_i\in (0,2\varphi_k)=(0,4\pi/k)$, and $\alpha_1+\ldots+\alpha_k=2\pi$.
Then it follows that at least one angle $\alpha_j$ satisfies $2\pi/k\leq\alpha_j<4\pi/k$.
This implies that $d(q_j,q_{j+1})=2\,r\sin(\alpha_j/2)\geq 2\,r\sin(\pi/k)$,
and so $d_M(S)\geq D(C_j)\geq d(q_j,q_{j+1})\geq 2\,r\sin(\pi/k)$, which proves the statement.
\end{proof}

We can now state the main results of this section, 
asserting that for any $k$-rotationally symmetric planar convex body,
the associated standard $k$-partition is a minimizing $k$-partition for the maximum relative diameter functional,
and also a minimizing $k$-subdivision when $k\leq 6$. Both results are immediate consequences of the previous Lemmata.

\begin{theorem}
\label{th:minimizing}
Let $C$ be a $k$-rotationally symmetric planar convex body, and $P_k$ the associated standard $k$-partition of $C$.
Let $P$ be any $k$-partition of $C$. Then, $d_M(P)\geq d_M(P_k)$.
\end{theorem}

\begin{proof}
From Lemma~\ref{le:dM}, we have that $d_M(P_k)=\max\{R,\,2\,r\sin(\pi/k)\}$,
where $R$ and $r$ are the circumradius and the inradius of $C$.
By using Lemmata~\ref{le:cota1} and~\ref{le:cota2}, we conclude that $d_M(P)\geq d_M(P_k)$.
\end{proof}

\begin{theorem}
\label{th:minimizing-subdivision}
Let $C$ be a $k$-rotationally symmetric planar convex body, and $P_k$ the associated standard $k$-partition of $C$,
with $k\leq 6$.
Let $S$ be any $k$-subdivision of $C$. Then $d_M(S)\geq d_M(P_k)$.
\end{theorem}

\begin{proof}
Taking into account Lemmata~\ref{le:dM}, ~\ref{le:cota1-subd} and ~\ref{le:cota2}, the result follows.
\end{proof}

We point out that the previous Theorem~\ref{th:minimizing-subdivision} is the best result
we can obtain for minimizing $k$-subdivisions, since it does not hold for $k\geq 7$,
as shown in the following examples.

\begin{example}
\label{ex:7}
In Remark~\ref{re:7} we considered a regular heptagon $H$ with circumradius $R=1$, and described a $7$-subdivision $S$ with $d_M(S)=0.9892$. As $d_M(P_7)=R=1$ for the standard $7$-partition $P_7$ associated to $H$ (see Example~\ref{ex:rR}), it follows that $P_7$ is not a minimizing $7$-subdivision for the maximum relative diameter.
\end{example}

\begin{example}
\label{ex:8}
Consider the unit circle $C$ and the $8$-subdivision $S$ shown in Figure~\ref{fig:8}, given by an inner subset $C_1$ bounded by a circle of radius $0.43$, and seven congruent subsets $C_2,\ldots,C_8$ determined by $7$-symmetric segments joining $\ptl C$ and $\ptl C_1$. Straightforward computations yield that $D(C_1)=0.86$ and $D(C_2)=0.86$, and so $d_M(S)=0.86$. Thus the corresponding standard $8$-partition $P_8$ is not a minimizing $8$-subdivision for the circle, since $d_M(P_8)=R(C)=1$, see Example~\ref{ex:dMcircle}.
\end{example}

\begin{figure}[h]
    \includegraphics[width=0.7\textwidth]{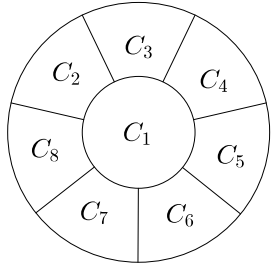}\\
  \caption{A $8$-subdivision of the circle}\label{fig:8}
\end{figure}

\begin{example}
The following idea suggests that when $k\in\nn$ is large enough,
the corresponding standard $k$-partition will not be a minimizing $k$-subdivision
for the maximum relative diameter.
Consider the unit circle, which satisfies $d_M(P_k)\geq1$ for any $k\in\nn$, $k\geq3$, where $P_k$ is the standard $k$-partition, see Example~\ref{ex:dMcircle}. It is clear that, when $k\geq k_0$ for $k_0\in\nn$ large enough,
we can divide the unit circle into $k$ \emph{small} subsets, each one with diameter strictly less than $1$,
thus obtaining a $k$-subdivision with maximum relative diameter strictly less than $1$.
This implies that $P_k$ is not minimizing for $k\geq k_0$.
This reasoning can be probably applied for regular $k$-gons when $k$ is large enough (and then close to be a circle).
In Section~\ref{sec:large} we shall study the minimizing $k$-subdivision for a planar convex body,
for large values of $k\in\nn$.
\end{example}

\begin{remark}
We remark that Theorems~\ref{th:minimizing} and~\ref{th:minimizing-subdivision}
were already obtained in \cite{trisecciones} for the case $k=3$.
\end{remark}

\subsection{Uniqueness of the minimizing $k$-partition}
\label{subsec:uniqueness}
The uniqueness of the solution is also an interesting question for these kinds of optimization problems.
In our setting, where the considered functional is the maximum relative diameter,
the uniqueness of the minimizing $k$-partition does not hold,
as we can see in the example from Figure~\ref{fig:non}.
For the regular hexagon $H$, we exhibit three different minimizing $6$-partitions:
the associated standard $6$-partition $P_6$ (which is minimizing by Theorem~\ref{th:minimizing}),
and two others $6$-partitions obtained by slight modifications of $P_6$,
keeping invariant the value of the maximum relative diameter $d_M(P_6,H)$.
Other similar examples can be constructed for the regular hexagon and, in general,
for any $k$-rotationally symmetric planar convex body (see Figure~\ref{fig:non-circle}).
This shows that the minimizing $k$-partition for any set of our class is not unique,
being not possible to obtain a complete classification of all of them.

\begin{figure}[h]
  \includegraphics[width=0.7\textwidth]{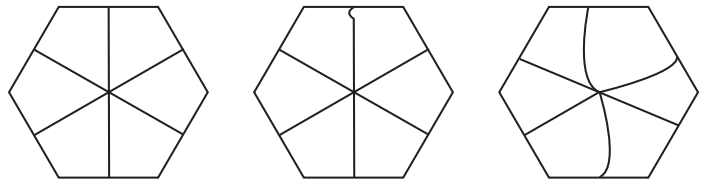}\\
  \caption{Three different minimizing $6$-partitions for the regular hexagon.
  The first one is the corresponding standard $6$-partition $P_6$, and the others are slight modifications.}
  \label{fig:non}
\end{figure}

\begin{figure}[h]
  \includegraphics[width=0.7\textwidth]{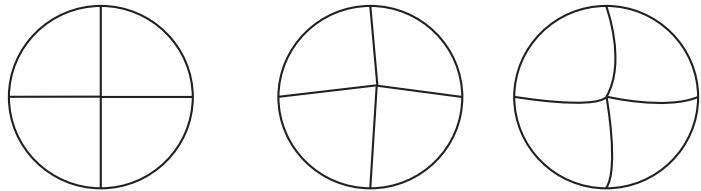}\\
  \caption{Three different minimizing $4$-partitions for the circle.
  The first one is the corresponding standard $4$-partition $P_4$, and the others are slight modifications.}
  \label{fig:non-circle}
\end{figure}

However, some necessary conditions for being a minimizing $k$-partition
can be derived from a deeper analysis of Lemmata~\ref{le:cota1} and~\ref{le:cota2}.
More precisely, let $C$ be a $k$-rotationally symmetric planar convex body,
with center of symmetry $p$, and let $P$ be a minimizing $k$-partition.
From Theorem~\ref{th:minimizing} and Lemma~\ref{le:dM} we have that
$d_M(P)=d_M(P_k)=\max\{R,2\,r\sin(\pi/k)\}$,
where $P_k$ is the standard $k$-partition associated to $C$,
and $R$ and $r$ are the circumradius and the inradius of $C$, respectively.
By discussing the two different possibilities, we shall obtain conditions that
the minimizing $k$-partition $P$ must satisfy.
If $d_M(P_k)=R$, from the proof of Lemma~\ref{le:cota1} we can deduce that
the common point of $P$ has to coincide necessarily with $p$
(otherwise, $d_M(P)$ will be \emph{strictly greater} than $R$).
On the other hand, if $d_M(P_k)=2\,r\sin(\pi/k)$, an analysis of the equality case
in the proof of Lemma~\ref{le:cota2} will yield that the endpoints of $P$
have to coincide with the endpoints of $P_k$.

\section{Optimal sets}
\label{se:optimal}

The search for the optimal sets in geometric optimization problems is, in general, a hard question.
For some problems, it is only possible to obtain some properties of the solutions,
but not a complete description of them.
However, for the problem considered in this paper, we will see that we can give
a precise characterization of the optimal set,
which in this setting corresponds with the $k$-rotationally symmetric planar convex body
attaining the minimum value for the maximum relative diameter.

Our initial motivation resides in the following question:
given a $k$-rotationally symmetric planar convex body $C$,
we know from Theorem~\ref{th:minimizing} that
the lowest value for the maximum relative diameter among all $k$-partitions of $C$
is achieved by the standard $k$-partition $P_k(C)$ associated to $C$,
and it is equal to $d_M(P_k(C),C)$.
Then, a natural further step is investigating for the minimum of such lowest values,
among all sets in the class $\mathcal{C}_k$. Equivalently, we are interested in finding
\[
\min\{d_M(P_k(C),C): C\in\mathcal{C}_k\},
\]
which will consequently determine the optimal body in $\mathcal{C}_k$.

However, in this setting it is more convenient to find
\[
\min\bigg\{\frac{d_M(P_k(C),C)^2}{A(C)}: C\in\mathcal{C}_k\bigg\},
\]
where $A(C)$ denotes the area of $C$, since the quotient
\begin{equation}
\label{eq:quotient}
\frac{d_M(P_k(C),C)^2}{A(C)}
\end{equation}
is invariant under dilations about the center of symmetry of $C$,
and therefore the optimal body will be independent from the enclosed area.

The following Theorem~\ref{th:optimal} provides the characterization of the optimal body in $C_k$,
for each $k\in\nn$, $k\geq 3$. 

\begin{theorem}
\label{th:optimal}
Among all $k$-rotationally symmetric planar convex bodies for fixed $k\in\nn$, $k\geq 3$,
the minimum value in \eqref{eq:quotient} is uniquely attained
by the intersection of the unit circle with the regular $k$-gon
with inradius $1/(2\sin(\pi/k))$, up to dilations.
\end{theorem}

\begin{proof}
In order to obtain the minimum value in~\eqref{eq:quotient},
we can fix either the value of the maximum relative diameter or the area of the set.
We shall proceed by fixing the maximum relative diameter,
and so we shall \emph{maximize} the area of the set.

Consider an arbitrary set $C$ in $\mathcal{C}_k$ with $d_M(P_k(C),C)=1$,
where $P_k(C)$ is the standard $k$-partition associated to $C$.
From Lemma~\ref{le:dM}, it follows that $R(C)\leq 1$ and $2\,r(C)\,\sin(\pi/k)\leq 1$,
where $R(C)$ and $r(C)$ are the circumradius and the inradius of $C$, respectively.
The first inequality implies that $C$ is contained in the unit circle,
and the second one implies that $r(C)\leq 1/(2\sin(\pi/k))$,
which gives that $C$ is contained in the regular $k$-gon with inradius $1/(2\sin(\pi/k))$,
due to the convexity of $C$.
Therefore, $C$ will be contained in the corresponding intersection,
and the set with maximal area under these conditions will be necessarily that intersection.
\end{proof}

\begin{remark}
The previous Theorem~\ref{th:optimal} characterizes the optimal body in $\mathcal{C}_k$ for our problem,
when considering $k$-partitions, for any $k\in\nn$, $k\geq 3$.
We can also study the same question for $k$-subdivisions.
In that case, taking into account Theorem~\ref{th:minimizing-subdivision},
analogous reasonings will yield to the same characterization of the optimal bodies only when $k\leq 6$.
\end{remark}

\subsection{Description of the optimal bodies}
\label{re:description}
We shall now determine the optimal bodies for each $k\in\nn$, $k\geq 3$,
by analyzing the corresponding intersection described in Theorem~\ref{th:optimal} in each particular case.
The precise optimal bodies are depicted in Figure~\ref{fig:optimals}.

For $k\in\{3,5\}$, we have that the value of the inradius $1/(2\sin(\pi/k))$ of the coresponding regular $k$-gon
is strictly less than $1$, and its circumradius is strictly greater than $1$ (see Example~\ref{ex:rR}).
This implies that the regular $k$-gon actually intersects the unit circle,
obtaining in these cases a sort of regular $k$-gons with \emph{rounded corners} as optimal bodies.

For $k=4$, the square with inradius $1/(2\sin(\pi/4))$ has circumradius equal to $1$.
Hence, that square is completely contained in the unit circle and provides the optimal body in this case.

Finally, for $k\geq6$, the inradius $1/(2\sin(\pi/k))$ of the regular $k$-gon $E_k$
is greater than or equal to $1$, which implies that the unit circle is entirely contained in $E_k$.
Therefore, the optimal bodies coincide with the unit circle for $k\geq 6$.

\begin{figure}[h]
  \includegraphics[width=0.9\textwidth]{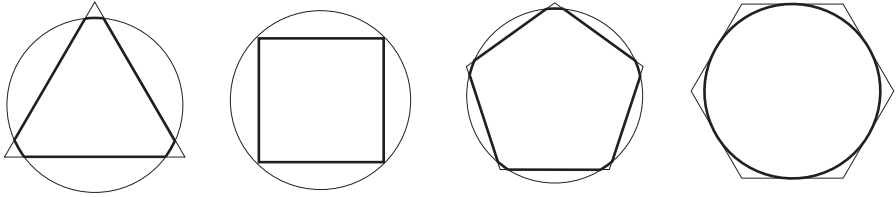}\\
  \caption{Optimal bodies for $k=3$, $k=4$, $k=5$, and $k\geq6$.}
  \label{fig:optimals}
\end{figure}%

\begin{remark}
The optimal body for $k=3$ was already obtained in~\cite[Th.~4.7]{trisecciones} by means of a constructive procedure,
and is a solution for some others geometric optimization problems, see~\cite[\S.~4]{hss}.
\end{remark}

\section{$k$-subdivisions for large values of $k$}
\label{sec:large}

In Section~\ref{sec:main} we have proved that, for each $k$-rotationally symmetric planar convex body,
the associated standard $k$-partition is a minimizing $k$-partition for the maximum relative diameter,
for any $k\in\nn$, $k\geq3$ (Theorem~\ref{th:minimizing}), and also a minimizing $k$-subdivision for $k\in\{3,4,5,6\}$ (Theorem~\ref{th:minimizing-subdivision}).
In this section, we will go further with the study of the $k$-subdivisions problem,
by considering the asymptotic situation 
for large values of $k\in\nn$, and
searching for a minimizing $k$-subdivision for the maximum relative diameter.

We have already seen that, in general, the standard $k$-partition is not minimizing when $k\geq 7$
(for instance, see Examples~\ref{ex:7} and~\ref{ex:8}).
The reasonings in this section will lead us to conjecture that an optimal $k$-subdivision,
when the number of regions $k$ is large enough, is provided by a \emph{hexagonal tiling}
(up to some subsets touching the boundary of the original set).
This kind of tilings have already appeared in the literature during the last years,
specially from the isoperimetric point of view
(see \cite{Hales} for the least-perimeter subdivision of the plane,
or \cite{bubbles} for the structure of the solutions of the planar soap bubble problem for a large number of regions).
We point out that the assumption of rotational symmetry can be removed,
and so we shall focus on general planar convex bodies.

The following lemma, whose proof relies on the classical isodiametric inequality \cite{bieberbach},
provides a first lower bound for the maximum relative diameter,
when considering $k$-subdivisions of an arbitrary planar convex body.

\begin{lemma}
\label{le:circles}
Let $C$ be a planar convex body, and let $S$ be a $k$-subdivision of $C$, for $k\in\nn$.
Then,
\begin{equation}
\label{eq:circles}
d_M(S,C)\geq \sqrt{\frac{A(C)}{k}}\sqrt{\frac{4}{\pi}}.
\end{equation}
\end{lemma}

\begin{proof}
Let $C_1,\ldots,C_k$ be the subsets of $C$ given by $S$, and call $d=d_M(S,C)$.
By using the isodiametric inequality~\cite{bieberbach}, since $D(C_i)\leq d$,
we have that $\displaystyle{A(S_i)\leq\frac{\pi}{4}\,d^2}$, for $i=1,\ldots,k$. Then,
\[
A(C)=\sum_{i=1}^k\,A(C_i)\leq k \frac{\pi}{4}\,d^2,
\]
which yields the statement.
\end{proof}

\begin{remark}
We recall that the equality in the classical isodiametric inequality holds for circles,
and so, the equality in~\eqref{eq:circles} cannot be attained,
since it is not possible to consider a subdivision given by circles
due to the overlappings of the corresponding subsets.
Therefore, \eqref{eq:circles} actually provides a \emph{strict} lower bound
for the maximum relative diameter functional.
\end{remark}

The following Lemma~\ref{le:cotainferior} will allow to improve
the lower bound obtained in the previous Lemma~\ref{le:circles}, for large values of $k$.

\begin{lemma}
\label{le:cotainferior}
Let $C$ be a planar convex body, and let $S$ be a $k$-subdivision of $C$ with $d_M(S,C)=d$.
Then,
\begin{equation}
\label{eq:cc}
A(C)\leq 0.688452\,k\,d^2 + d\,P(C),
\end{equation}
where $P(C)$ denotes the perimeter of $C$.
\end{lemma}

\begin{proof}
Call $C_1,\ldots,C_k$ the subsets of $C$ given by $S$.
Let $\widetilde{C_i}$ be the convex hull of $C_i$.
Then $D(\widetilde{C_i})=D(C_i)$ for $i=1,\ldots,k$,
and $C$ is covered by the union of these convex hulls.
Since the possible overlappings can be eliminated by dividing the intersection regions by proper line segments,
we can then assume that $C_1,\ldots,C_k$ are convex polygons up to some arcs contained in $\ptl C$.
By replacing each of these arcs by the line segment determined by its endpoints,
we shall obtain convex polygons $C_i'\subset C_i$, $i=1,\ldots,k$, whose union is a convex polygon $C'$ contained in $C$.
It is clear that $C\backslash C'$ is contained in the union of stripes parallel to $\ptl C'$ and breadth $d$.
This implies
\begin{equation}
\label{eq:ineq}
A(C)\leq A(C')+d\,P(C')\leq A(C')+d\,P(C)=\sum_{i=1}^k A(C_i')+d\,P(C).
\end{equation}

The convex polygons $C_1',\ldots,C_k'$ determine a planar graph with $k+1$ faces
(together with the unbounded component $C_0'$). If $m$ and $n$ denote the number of edges and vertices,
by Euler's theorem we have $(k+1)-m+n=2$.

Since each edge is incident to two vertices and the degree of each vertex
is at least 3 we obtain $\displaystyle 2m \geq 3n$.
It follows $1 = k - m + n \leq k -m + 2/3 m = k - m/3$ and so
\[
m \leq 3 k - 3.
\]

Let $f_j$ be the number of polygons with $j$ edges in $\{C_1',\ldots,C_k'\}$, and $s$ the number of edges of
$C_0$. Then $$2\,m=s+\sum_j f_j\,j$$ and so $$\sum_jf_j\,j< s+\sum_j f_j\,j=2\,m\leq 6\,k-6\leq 6\,k.$$

On the other hand, let $m_j$ be the maximal area of a polygon of $j$ edges with diameter one.
It is known (see~\cite{m}) that $m_j$ is attained by regular polygons for odd $j$, with
$$m_j=j/2\cos(\pi/j)\tan(\pi/(2j)).$$
Furthermore, the values for $m_4$, $m_6$ and $m_8$ are also numerically known (see~\cite{m}),
so that $m_3=0.433012$, $m_4=0.5$, $m_5=0.657163$, $m_6=0.6749814429$, $m_7=0.719740$, $m_8=0.7268684828$, $m_9=0.745619$.
Moreover, notice that $m_j\leq\pi/4$, for any $j\geq3$,
as the corresponding polygon will be contained in the circle of diameter one.
Then, if $C_i'$ is a polygon with $j$ edges, $A(C_i')\leq d^2\,m_j$, and so
\begin{equation}
\label{eq:larga}
\sum_{i=1}^k A(C_i')\leq \sum_j f_j\,m_j\,d^2= d^2\sum_{j}f_j\,m_j
\leq d^2\bigg(\sum_{j=3}^9 f_j\,m_j +\pi/4\sum_{j\geq10} f_j\bigg).
\end{equation}
We are interested in finding an upper bound for \eqref{eq:larga}.
By replacing $\sum_{j\geq10} f_j$ by $\widetilde{f}$ and calling $\widetilde{m}=\pi/4$,
and using linear programming, it is straightforward checking that the maximal value
for $f_3\,m_3+\ldots+f_9\,m_9+\widetilde{f}\,\widetilde{m}$ under the restrictions
$f_3+\ldots+f_9+\widetilde{f}=k$, $3f_3+\ldots+9f_9+10\widetilde{f}\leq 6\,k$, and $f_3,\ldots,f_9,\widetilde{f}\geq0$
is given by $k(m_5+m_7)/2=0.688452\,k$, being attained by $f_5=f_7=k/2$ and $f_3=\ldots=f_9=\widetilde{f}=0$.

Using that bound in \eqref{eq:larga}, we conclude that $\sum_{i=1}^k A(C_i')\leq 0.6884\,k\,d^2$,
which yields the statement taking into account \eqref{eq:ineq}.
\end{proof}

\begin{corollary}
\label{co:cota}
Let $C$ be a planar convex body, and let $S$ be a $k$-subdivision of $C$. Then
\begin{equation}
\label{eq:O}
d_M(S,C)\geq \sqrt{\frac{A(C)}{k}}\sqrt{\frac{1}{0.688452}}-O(1/k).
\end{equation}
\end{corollary}

\begin{proof}
The statement is a direct consequence of \eqref{eq:cc}.
\end{proof}

\begin{remark}
\label{re:bestbound}
For large values of $k\in\nn$, \eqref{eq:O} gives an improvement of the lower bound in \eqref{eq:circles}.
However, such a bound is \emph{strict}, since a planar tiling by regular pentagons and regular heptagons cannot be realized.
\end{remark}

In the following Lemma~\ref{le:hexagons} we shall consider a particular $k$-subdivision for any planar convex body
and any $k\in\nn$, induced by a proper tiling of the plane by regular hexagons, see Figure~\ref{tiling}.
We shall estimate the maximum relative diameter $d_M$ for these $k$-subdivisions,
obtaining an upper bound for the minimal value of $d_M$,
which we conjecture it is the optimal one for large values of $k\in\nn$.

\begin{lemma}
\label{le:hexagons}
Let $C$ be a planar convex body $C$. For each $k\in\nn$, there exists a $k$-subdivision $S_k$ such that
\begin{equation}
\label{eq:cotahexagonos}
d_M(S_k,C)\leq \sqrt{\frac{A(C)}{k}}\sqrt{\frac{8}{3\sqrt{3}}} + O\bigg(\frac{1}{k}\bigg).
\end{equation}
\end{lemma}

\begin{proof}
Call $H$ the regular hexagon with unit diameter, and so $A(H)=3\sqrt{3}/8$.
For any $d>0$, let $L_d$ be the planar lattice with fundamental cell $d\,H$.
For any $k\in\nn$, let $d_k$ be the unique positive solution of
\begin{equation}
\label{eq:ecuacion}
k\,A(d\,H)=A(C)+d\,P(C)+\pi\,d^2,
\end{equation}
where $P(C)$ denotes the perimeter of $C$.
It is clear that $C$ is covered by the union of the fundamental cells of $L_{d_k}$
with nonempty intersection with $C$.
Call $n$ the number of such fundamental cells, which will be contained in the outer parallel body $C+d_k\,B$,
where $B$ is the planar unit ball. Then, Steiner's formula and \eqref{eq:ecuacion} give
$$n\,A(d_k\,H)\leq A(C+d_k\,B)=A(C)+d_k\,P(C)+\pi\,d_k^2=k\,A(d\,H),$$
and so $n\leq k$. This means that the lattice $L_{d_k}$ induces a $k$-subdivision $S_k$ of $C$
(if $n<k$, we can divide a fundamental cell until having a decomposition of $C$ into $k$ connected subsets)
with $d_M(S_k,C)\leq d_k$. This fact, together with a straightforward estimate of $d_k$ by using \eqref{eq:ecuacion},
proves the statement.
\end{proof}

\begin{remark}
We stress that the $k$-subdivision $S_k$ from Lemma~\ref{le:hexagons} consists of
regular hexagons of diameter $d_k$ in the interior of the set,
and portions of such hexagons when meeting the boundary of the set.
\end{remark}

\begin{figure}[h]
    \includegraphics[width=0.25\textwidth]{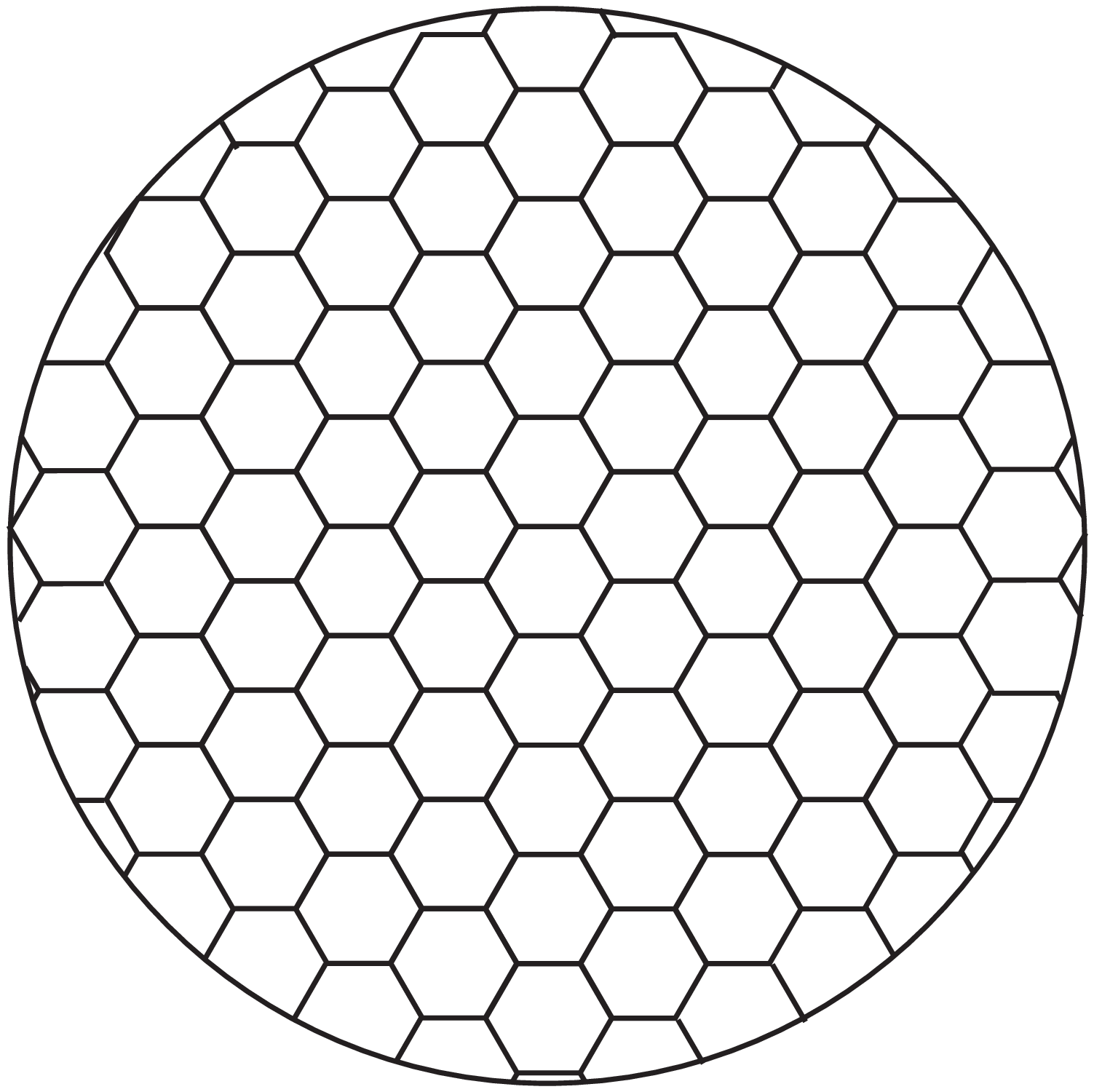}\\
  \caption{A subdivision of the circle given by regular hexagons}
  \label{tiling}
\end{figure}

The previous Lemma~\ref{le:hexagons} gives an upper bound
for the minimum possible value of the maximum relative diameter for large enough values of $k\in\nn$.
However, we think that such a bound is the optimal one due to the following fact:
taking into account Lemma~\ref{le:circles} and using the covering density \cite{densities}
(recall that the subdivision by circles would produce overlappings of the subsets),
it is easy to check that the bound from \eqref{eq:cotahexagonos} can be obtained
from the bound given in \eqref{eq:circles}, when $k$ is large enough.
This means, in some sense, that the hexagonal tiling corrects the defects appearing in the decomposition by circles.
In view of this, we finish this section with the following Conjecture~\ref{conj:infinito},
regarding the minimizing $k$-subdivisions for a planar convex set and large values of $k\in\nn$.

\begin{conjecture}
\label{conj:infinito}
Let $C$ be a planar convex set, and let $S$ be any $k$-subdivision of $C$.
For large enough values of $k\in\nn$, the minimal value for $d_M(S,C)$
is provided by the $k$-subdivision described in Lemma~\ref{le:hexagons},
given by regular hexagons of fixed diameter in the interior of $C$
and portions of such hexagons when meeting the boundary of $C$.
\end{conjecture}

\noindent {\bf Acknowledgements.} 
The authors would like to thank the Instituto de Matem\'aticas de la Universidad de Sevilla (IMUS), where part of this work was done.
The first author is partially supported by the research projects MTM2010-21206-C02-01 and MTM2013-48371-C2-1-P,
and by Junta de Andaluc\'ia grants FQM-325 and P09-FQM-5088.
The third author is partially supported by MINECO (Ministerio de Economía y Competitividad)
and FEDER (Fondo Europeo de Desarrollo Regional) project MTM2012-34037.

\end{document}